\documentclass[journal,twoside,web]{ieeecolor}

\usepackage{generic}
\usepackage{cite}
\usepackage{booktabs}
\usepackage{caption}
\usepackage{amsmath,amssymb,amsfonts}
\usepackage{graphicx}
\usepackage{textcomp}
\usepackage{mathtools}
\usepackage{xcolor}
 
\usepackage{amsthm}
\usepackage{fancyhdr}
\usepackage{lcsys}

\makeatletter

\newtheorem{theorem}{Theorem}[section]

\newtheorem{proposition}[theorem]{Proposition}

\newtheorem{definition}[theorem]{Definition}
\newtheorem{remark}[theorem]{Remark}
\newtheorem{assum}[theorem]{Assumption}

\fancypagestyle{FirstPage}{
	\lfoot{		\begin{minipage}{\textwidth}
		\footnotesize
			\copyright 2022 IEEE. Personal use of this material is permitted.  Permission from IEEE must be obtained for all other uses, in any current or future media, including reprinting/republishing this material for advertising or promotional purposes, creating new collective works, for resale or redistribution to servers or lists, or reuse of any copyrighted component of this work in other works.
	\end{minipage}} 
}
\begin{document}


\title{Recursively Feasible Data-Driven Distributionally Robust Model Predictive Control with Additive Disturbances}

\author{Christoph Mark and Steven Liu 
	\thanks{Institute  of  Control  Systems,  Department  of  Electrical  and  Computer
		Engineering, University of Kaiserslautern, Erwin-Schrödinger-Str. 12, 67663
		Kaiserslautern, Germany, {\tt\small mark|sliu@eit.uni-kl.de}}%
}

\maketitle

\thispagestyle{empty}

\begin{abstract}
	In this paper we propose a data-driven distributionally robust Model Predictive Control framework for constrained stochastic systems with unbounded additive disturbances. Recursive feasibility is ensured by optimizing over a linearly interpolated initial state constraint in combination with a simplified affine disturbance feedback policy. We consider a moment-based ambiguity set with data-driven radius for the second moment of the disturbance, where we derive a minimum number of samples in order to ensure user-given confidence bounds on the chance constraints and closed-loop performance. The paper closes with a numerical example, highlighting the performance gain and chance constraint satisfaction based on different sample sizes.
\end{abstract}

\section{Introduction}
\thispagestyle{FirstPage}
Model predictive control (MPC) is an optimization-based control method capable of dealing with general constraints and performance criteria \cite{kouvaritakis2016model}. In many MPC applications, decisions must be made under uncertainty, leading to robust \cite{mayne2005robust} or stochastic MPC \cite{farina2016stochastic} approaches. While robust MPC requires an a-priori known bound on the disturbance to satisfy hard constraints, stochastic MPC relaxes this requirement by using knowledge of the underlying moment information of the stochastic disturbance to soften the constraints as chance constraints \cite{hewing2018stochastic}. In practice, exact moment information is rarely available and must be estimated from data, which leads to distributionally robust optimization (DRO) methods \cite{rahimian2019distributionally}. 

The main idea of DRO is to define a so-called ambiguity set that contains plausible variations of the empirically estimated distribution (or moment information). In the DRO literature, a rough distinction is made between moment-based and discrepancy-based sets \cite{rahimian2019distributionally}, both of which are considered in several works on distributionally robust MPC (DR-MPC), e.g., the Wasserstein discrepancy measure is used in \cite{mark2020stochastic, mark2021data, lu2020soft, coulson2021distributionally}, the total variation discrepancy measure in \cite{dixit2022distributionally}, whereas the authors of \cite{van2015distributionally, coppens2021data} rely on moment-based ambiguity sets.

\subsubsection*{Contribution} In this paper we study a data-driven approach to DR-MPC for linear systems with additive unbounded disturbances under moment-based ambiguity sets. We use a simplified disturbance feedback (SADF) control parameterization \cite{zhang2020stochastic} with the intention to reduce the number of decision variables of the MPC optimization problem.
This further allows us to replace the expected value cost function and the chance constraints with their distributionally robust surrogates, where the constraints are cast as second order cone (SOC) constraints \cite{Calafiore}. To ensure recursive feasibility we employ a similar technique as suggested by \cite{kohler2022recursively}, where we optimize the initial state constrained on a line between a guaranteed feasible solution (the shifted optimal solution from the previous time step) and the state feedback initialization. In Proposition \ref{prop:covariance} we derive a minimum number of disturbance samples, such that the ambiguity set contains the true moments with high probability. In comparison to \cite{coppens2020data} we thus require fewer samples by assuming a known first moment. This enables us to derive rigorous theoretical properties, such as an expected cost decrease, an asymptotic average performance bound and satisfaction of conditional chance constraint.

Coppens and Patrinos \cite{coppens2021data} recently proposed a comparable DR-MPC framework under the assumption of bounded disturbances, where the main difference lies in the handling of the distributionally robust chance constraints and in ensuring recursive feasibility. Van Parys et. al. \cite{van2015distributionally} proposed a DR-MPC framework under a moment-based ambiguity set, which is however not data-driven. The distributionally robust chance constraints are approximated using a conditional Value-at-Risk measure. Similar to \cite{coppens2021data, mark2020stochastic} we show in a numerical example the closed-loop performance and constraint satisfaction rate for different samples sizes $N_s$. 

The paper is organized as follows.
In Section \ref{sec:preliminaries} we introduce the problem setting and our result on a data-driven ambiguity radius. Section \ref{sec:MPC} discusses the SADF control parameterization, the treatment of the distributionally robust chance constraints and cost function, the required terminal ingredients and initial conditions for the DR-MPC problem. Section \ref{sec:theoretical} is devoted to the theoretical analysis of the DR-MPC problem and Section \ref{sec:example} to a numerical example. The paper closes with some concluding remarks.
\section{Preliminaries}
\label{sec:preliminaries}
\subsection{Notation}
A probability space is defined by the triplet $(\Omega, \mathcal{F}, \mathbb{P})$, where $\Omega$ is the sample space, $\mathcal{F}$ the $\sigma$-algebra on $\Omega$ and $\mathbb{P}$ the probability measure on $(\Omega, \mathcal{F})$.  Given an event $E_1$ we define the probability occurrence as $\mathbb{P}(E_1)$ and the conditional probability given $E_2$ as $\mathbb{P}(E_1 | E_2)$. For a random variable $w$, we define the expected value as $\mathbb{E}\{w\}$, whereas the conditional expectation of $w$ conditional to a random variable $x$ is denoted as $\mathbb{E}\{w | x\}$. The weighted 2-norm w.r.t. a positive definite matrix $Q = Q^\top$ is $\Vert x \Vert_Q^2 = x^\top Q x$. Positive definite and semidefinite matrices are indicated as $A>0$ and $A\geq0$, respectively. The ceiling function is denoted as $\lceil \cdot \rceil$. The pseudo inverse of a matrix $A$ is denoted as $A^\dagger$.
\subsection{Problem setting}
\label{sec:problem}
In this paper, we consider a discrete linear time-invariant system
\begin{align}
	x(k+1) = A x(k) + B u(k) + Ew(k) \label{eq:dynamics}
\end{align}
with the state $x \in \mathbb{R}^{n_x}$, input $u \in \mathbb{R}^{n_u}$, disturbance $w \in \mathbb{R}^{n_w}$ and matrices $A$, $B$ and $E$ are of conformal dimensions. The stochastic disturbance $w(k)$ is assumed to be i.i.d. for all $k \in \mathbb{N}$ that follows an unknown distribution (push-forward measure) $\mu^*$. We assume that we have access to $i = 1,\ldots, N_s$ random samples of $w \sim \mu^*$, which we denote as $\hat{w}^i$. We further assume that the matrix pair $(A,B)$ is stabilizable, the matrix $E$ has full column rank and that perfect state measurement is available at each time instant $k$. To distinguish between closed-loop and open-loop we introduce the prediction dynamics over a time horizon of length $N \in \mathbb{N}$
\begin{align*}
	\bar{x}_k = \bar{A} x_{0|k} + \bar{B} \bar{u}_k + \bar{E} \bar{w}_k,
\end{align*}
where $\bar{x}_k = [x^\top_{0|k}, x^\top_{1|k}, \ldots, x^\top_{N|k}]^\top \in \mathbb{R}^{(N+1) n_x}$ denotes the state sequence, $\bar{u}_k = [u^\top_{0|k}, \ldots, u^\top_{N-1|k}]^\top \in \mathbb{R}^{N n_u}$ the input sequence and $\bar{w}_k = [w^\top_{0|k}, \ldots, w^\top_{N-1|k}]^\top \in \mathbb{R}^{N n_w}$ the disturbance sequence. The matrices $\bar{A} \in \mathbb{R}^{(N+1)n_x \times n_x}$, $\bar{B} \in \mathbb{R}^{(N+1)n_x \times N n_u}$ and $\bar{E} \in \mathbb{R}^{ (N+1) n_x \times Nn_w }$ are given as 
{\scriptsize
	\begin{align*}
		&\bar{A} \coloneq \begin{bmatrix}
			I \\
			A \\
			\vdots \\
			A^N
		\end{bmatrix}, \bar{B} \coloneqq \begin{bmatrix}
			0 & 0 & \dots & 0 \\
			B & 0 & \dots & 0 \\
			AB & B & \dots & 0 \\
			\vdots & \ddots & \ddots & 0 \\
			A^{N-1}B & \dots & AB & B
		\end{bmatrix}, \\
	&\bar{E} \coloneqq \begin{bmatrix}
			0 & 0 & \dots & 0 \\
			E & 0 & \dots & 0 \\
			AE & E & \dots & 0 \\
			\vdots & \ddots & \ddots & 0 \\
			A^{N-1}E & \dots & AE & E
		\end{bmatrix}.
	\end{align*}
}
We impose chance constraints on the states and inputs
\begin{subequations}
	\begin{align}
		\mathbb{P}(h^\top_{t,r} \bar{x}_k &\leq 1 \: | \: x(k)) \geq p^x_{r} \: \: && t = 0, \ldots, N-1 \label{eq:constraints:state_chance}\\
		\mathbb{P}(l_{t,s}^\top \bar{u}_k &\leq 1 \: | \: x(k)) \geq p^u_{s} \: \: && t = 0, \ldots, N-1,
		\label{eq:constraints:input_chance}
	\end{align} \label{eq:constraints}%
\end{subequations}
where $h_{t,r} \in \mathbb{R}^{(N+1)n_x}$ and $l_{t,s} \in \mathbb{R}^{N n_u}$ denote the left-hand-side of the $r = 1,\ldots, n_r$ state and $s = 1, \ldots, n_s$ input half-space constraints and $p_r^x, p_s^u \in (0,1)$ are the required levels of chance constraint satisfaction. Given an initial value $x_{0|k}$ we opt to solve the following finite horizon stochastic optimal control problem
\begin{subequations}
	\label{eq:smpc_optimization}
	\begin{align}
		\!\min_{\bar{u}_k} & \quad \mathbb{E}_{\mu^*} \left\{ \Vert x_{N|k} \Vert_{P}^2 + \sum_{t=0}^{N-1} \Vert x_{t|k} \Vert_{Q}^2 + \Vert u_{t|k} \Vert_{R}^2  \bigg | x(k)\right\}  \label{eq:smpc:cost} \\
		\text{s.t.} & \quad \bar{x}_k = \bar{A} x_{0|k} + \bar{B} \bar{u}_k + \bar{E} \bar{w}_k \nonumber \\
		& \quad \mathbb{P}(h^\top_{t,r} \bar{x}_k \leq 1 \: | \: x(k)) \geq p^x_{r} \quad t = 0, \ldots, N-1 \label{eq:smpc:chance_constraints}\\
		& \quad \mathbb{P}(l_{t,s}^\top \bar{u}_k \leq 1 \: | \: x(k)) \geq p^u_{s} \quad \: t = 0, \ldots, N-1, \label{eq:smpc:input_chance_constraints}
	\end{align}
\end{subequations}
where $Q > 0$, $R > 0$ and $P > 0$ are positive definite symmetric weighting matrices and $P$ additionally satisfies the Lyapunov inequality
\begin{align}
	(A + B K)^\top P (A + B K)+ Q + K^\top R K \leq P \label{eq:lyapunov}
\end{align}
for some linear controller matrix $K \in \mathbb{R}^{n_u \times n_x}$.

Problem \eqref{eq:smpc_optimization} represents an infinite dimensional optimization problem due to the control input $u$ and the additive disturbance $w$. This issue will be tackled in Section \ref{sec:control_param} by using a SADF parameterization. Furthermore, the cost function \eqref{eq:smpc:cost} and chance constraints \eqref{eq:smpc:chance_constraints}-\eqref{eq:smpc:input_chance_constraints} are evaluated w.r.t. the true, but unknown distribution $\mu^*$. 
Thus, we formulate a distributionally robust optimization problem that uses a so-called ambiguity set $\mathcal{P}$, i.e. a set of probability distributions, where each $\mu \in \mathcal{P}$ lies within some distance to the sample covariance $\hat{\Sigma} = N_s^{-1} \sum_{i=1}^{N_s} \hat{w}^i (\hat{w}^i)^\top$ under the assumption that $\mathbb{E}_{\mu}\{w\} = 0$. In particular, the ambiguity set represents the uncertainty of the empirical estimator and is parameterized as in \cite{delage2010distributionally}
\begin{align}
	\label{eq:ambiguity_set}
	\mathcal{P} \coloneqq \left\{  \mu \in \mathcal{M} \  \middle\vert \begin{array}{l}
		\mathbb{E}_{\mu}\{ w \}  = 0 \\
		\mathbb{E}_{\mu}\{ w  w^\top \} \leq \kappa_\beta \hat{\Sigma} \}
	\end{array}\right\},
\end{align}
where $\mathcal{M}$ denotes the set of all probability distributions defined on $(\mathbb{R}^{n_w}, \mathcal{B})$ with  $\mathcal{B}$ the associated Borel $\sigma$-algebra of $\mathbb{R}^{n_w}$. For some confidence level $\beta \in (0,1)$ we define a constant $\kappa_\beta \geq 1$, such that $\mathbb{P}( \mu^* \in \mathcal{P} ) \geq 1-\beta$.
\subsection{Data-driven ambiguity set}
In the following we derive an explicit value for the constant $\kappa_\beta$ under the assumption of sub-Gaussianity of the random variables $w(k)$, which was similarly done by \cite{coppens2020data}. This extends the results from Delage and Ye \cite{delage2010distributionally}, who provide an explicit value $\kappa_\beta$ for bounded random variables.
\begin{definition}
	A random variable $\xi$ is sub-Gaussian with variance proxy $\sigma^2$ if $\mathbb{E}\{ \xi \} = 0$ and its moment generating function satisfies
	\begin{align*}
		\mathbb{E}\{ e^{\lambda \xi} \} \leq e^{\frac{\sigma^2 \lambda^2}{2}} \quad \forall \lambda \in \mathbb{R}.
	\end{align*}
	We denote this by $\xi \sim \text{subG}(\sigma^2)$.
\end{definition}
\begin{proposition}
	\label{prop:covariance}
	Let $w \in \mathbb{R}^{n_w}$ be a zero-mean sub-Gaussian random variable with $\mathbb{E}\{ w  w^\top \} = \Sigma$. Let $\{ \hat{w}^i \}_{i=1}^{N_s}$ be $N_s$ i.i.d. samples obtained from the true distribution of $w$ and define $\hat{\Sigma} = N_s^{-1} \sum_{i=1}^{N_s} \hat{w}_i (\hat{w}^i)^\top$ as the empirical covariance matrix. 
	Let $\epsilon \in (0,0.5)$, $\beta \in (0,1)$, $c_1(\sigma, \epsilon) = \sigma^2 / (1 - 2 \epsilon)$, $c_2(\beta, \epsilon, n_w) = n_w \log(1 + 2/\epsilon) + \log(2/\beta)$, then for all $N_s \in \mathbb{N}$ satisfying 
	\begin{align}
		N_s \geq \left\lceil 2 c_1 c_2 \left( 8 c_1 + 4 \sqrt{4  c_1^2 + c_1} + 1 \right) \right\rceil, \label{eq:thm:condition}
	\end{align}
	the covariance bound $\Sigma \leq \frac{1}{1 - \gamma(N_s, \beta/2)}\hat{\Sigma}$ holds with a probability of at least $1 - \beta$, where 
	\begin{align*}
		\gamma(N_s,\beta) \coloneqq c_1(\sigma, \epsilon) \left( \sqrt{ \frac{32 c_2(\beta, \epsilon, n_w)}{N_s}} + \frac{2 c_2(\beta, \epsilon, n_w)}{N_s} \right).
	\end{align*}
\end{proposition}
\begin{proof}
The proof follows from \cite[Thm. 8]{coppens2020data}. Define a random variable $\xi = \Sigma^{-1/2}w \sim \text{subG}(\sigma^2)$ such that
	\begin{align*}
		\mathbb{E}\{ \xi \} &= \Sigma^{-1/2} \mathbb{E}\{w\} = 0 \\
		\mathbb{E}\{\xi \xi^\top \} &= \Sigma^{-1/2} \mathbb{E}\{ w w^\top \} \Sigma^{-1/2} = I
	\end{align*}
	and let $\tilde{I} = N_s^{-1} \sum_{i=1}^{N_s} \hat{\xi}^i (\hat{\xi}^i)^\top$ be the empirical covariance matrix of $\xi$. Consider now the empirical covariance matrix $\hat{\Sigma} = N_s^{-1} \sum_{i=1}^{N_s} \hat{w}^i (\hat{w}^i)^\top$ of the actual random variable $w$, which, after substitution of $\hat{w}^i = \Sigma^{1/2} \hat{\xi}^i$ equals
	\begin{align}
		\hat{\Sigma} = \Sigma^{1/2}   \left[ N_s^{-1} \sum_{i=1}^{N_s} \hat{\xi}^i (\hat{\xi}^i)^\top \right] \Sigma^{1/2} = \Sigma^{1/2} \tilde{I} \Sigma^{1/2}. \label{eq:proof:covariance_reform}
	\end{align}
	From \cite[Lem. A.1]{hsu2012tail} we have with probability of at least $1-\beta$ that
	$\Vert \tilde{I} - I \Vert_2 \leq \gamma(N_s,\beta/2)$, which is equivalent to
	\begin{align}
		(1 - \gamma(N_s, \beta/2))I \leq \tilde{I} \leq (1 + \gamma(N_s, \beta/2)) I. \label{eq:proof:cov}
	\end{align}
	Since we are only interested in an upper bound for the covariance matrix $\Sigma$, e.g. as required by \eqref{eq:ambiguity_set}, we find from the left inequality in \eqref{eq:proof:cov} that
	\begin{align*}
		I \leq \frac{1}{1 - \gamma(N_s, \frac{\beta}{2})}\tilde{I} \overset{\eqref{eq:proof:covariance_reform}}{\Longrightarrow} \Sigma \leq \frac{1}{1 - \gamma(N_s, \frac{\beta}{2})}\hat{\Sigma}.
	\end{align*}
	where we used the fact that condition \eqref{eq:thm:condition} implies $\gamma(N_s, \beta/2) < 1$. Finally, condition \eqref{eq:thm:condition} follows from assuming that $1-\gamma(N_s, \beta/2) > 0$, which is a quadratic inequality in the sample size $\sqrt{N_s}$.
\end{proof}
For a fixed $\beta \in (0,1)$, the mapping from $\epsilon \mapsto N_s$ in condition \eqref{eq:thm:condition} is convex on the interval $\epsilon \in (0, 0.5)$. Thus, to obtain the smallest number $N_s$ satisfying \eqref{eq:thm:condition}, we solve a nonlinear (convex) optimization problem 
\begin{align}
	\epsilon^* = \underset{\epsilon \in (0, 0.5)}{\arg\min} \:\:\: 2 c_1 c_2 \left( 8 c_1 + 4 \sqrt{4  c_1^2 + c_1} + 1 \right). \label{eq:nonlinear_opt}
\end{align}
Finally, by setting $\kappa_\beta = 1/(1 - \gamma(N_s, \beta/2))$ we obtain from optimization problem \eqref{eq:nonlinear_opt} and Proposition \ref{prop:covariance} an explicit number of samples $N_s$ to ensure that the ambiguity set \eqref{eq:ambiguity_set} contains the true distribution $\mu^*$ with $1-\beta$ confidence. 
\begin{remark}
	The result of Proposition \ref{prop:covariance} is a special case of \cite[Thm. 8]{coppens2020data} with known first moment information. As a consequence we require fewer samples compared to the mean and variance ambiguity set proposed by \cite{coppens2020data} to achieve the $1-\beta$ confidence of the ambiguity set \eqref{eq:ambiguity_set}.
\end{remark}
\section{Distributionally Robust MPC}
\label{sec:MPC}
\subsection{Control parameterization}
\label{sec:control_param}
We resort to a SADF parameterization \cite{zhang2020stochastic} of the form 
\begin{align}
	u_{i|k} = v_{i|k} + \sum_{t=0}^{i-1} M_{i-t|k} w_{t|k}, \label{eq:disturbance_feedback}
\end{align}
where $v_{i|k} \in \mathbb{R}^{n_u}$ is the predicted control input and $M_{i-t|k} \in \mathbb{R}^{n_u \times n_x}$ are feedback matrices, both of which are decision variables in the resulting MPC optimization problem. Thus, $u_{i|k}$ depends affinely on the past $i$ disturbance $w_{0|k}, \ldots, w_{i-1|k}$. To streamline the presentation we consider the matrix $\bar{M}_k \in \mathbb{R}^{N n_u \times N n_w}$ and the vector $\bar{v}_k \in \mathbb{R}^{N n_u}$
\begin{align*}
	&\bar{M}_k \coloneqq \begin{bmatrix}
		0 & 0 & \dots & 0 \\
		M_{1|k} & 0 & \dots & 0 \\
		\vdots & \ddots & \ddots & 0 \\
		M_{N-1|k} & \dots & M_{1|k} & 0
	\end{bmatrix}, \:
	\bar{v}_k \coloneqq \begin{bmatrix}
		v_{0|k} \\
		v_{1|k} \\
		\vdots \\
		v_{N-1|k}
	\end{bmatrix}
\end{align*}
such that $\bar{u}_k = \bar{v}_k + \bar{M}_k \bar{w}_k$. The vector $\bar{v}_k$ can be interpreted as the predicted nominal control input that corresponds to the predicted nominal state trajectory $\bar{z}_k  =\bar{A} x_{0|k} + \bar{B} \bar{v}_k$. 
\subsection{Distributionally robust chance constraints}
\label{sec:chance_constraints}
In the following we replace the individual chance constraints \eqref{eq:smpc:chance_constraints}-\eqref{eq:smpc:input_chance_constraints} with distributionally robust chance constraints of the form
\begin{subequations}
	\begin{align}
		& \inf_{\mu \in \mathcal{P}} \: \mathbb{P}(h^\top_{t,r} \bar{x}_k \leq  1 \: | \: x(k)) \nonumber \\
		&=\inf_{\mu \in \mathcal{P}} \: \mathbb{P} \bigg( \underbrace{\begin{bmatrix}
				\bar{w}_k^\top & 1
		\end{bmatrix}}_{\coloneqq d^\top}
		\underbrace{\begin{bmatrix}
				h^\top_{t,r} (\bar{B} \bar{M}_k + \bar{E}) \\
				h^\top_{t,r} \bar{z}_k
		\end{bmatrix}}_{\coloneqq \tilde x} \leq 1 \bigg \vert x(k) \bigg) \geq p_r^x
		\label{eq:dr_chance_constraint} \\
		&\inf_{\mu \in \mathcal{P}} \:  \mathbb{P}(l_{t,s}^\top \bar{u}_k \leq 1 \: | \: x(k)) \geq p^u_{s}	\nonumber\\
		&=\inf_{\mu \in \mathcal{P}} \: \mathbb{P} \bigg( \begin{bmatrix}
			\bar{w}_k^\top & 1
		\end{bmatrix}
		\underbrace{\begin{bmatrix}
				l^\top_{t,s} \bar{M}_k \\
				l^\top_{t,s} \bar{v}_k
		\end{bmatrix}}_{\coloneqq \tilde u} \leq 1 \bigg \vert x(k)\bigg) \geq p_s^u, \label{eq:dr_input_constraint}
	\end{align}
\end{subequations}
where we substituted the expressions for $\bar{x}_k$, $\bar{u}_k$ and $\bar{z}_k$. By definition of the ambiguity set \eqref{eq:ambiguity_set} it holds that $\mathbb{E}_{\mu}\{ w\} = 0$ and $\sup_{\mu \in \mathcal{P}} \mathbb{E}_{\mu} \{ w w^\top \} = \kappa_\beta \hat{\Sigma}$. Thus, $\mathbb{E}_\mu \{ d \} =  [0 \quad 1]^\top$ and $\sup_{\mu \in \mathcal{P}} \mathbb{E}_{\mu} \{ (d-\mathbb{E}_\mu \{ d \}) (d-\mathbb{E}_\mu \{ d \})^\top \} = \scriptsize\begin{bmatrix}
	\kappa_\beta \hat{\Sigma}_N  & 0 \\
	0 & 0
\end{bmatrix}$ with $\hat{\Sigma}_N = I_{N} \otimes \hat{\Sigma}$. It remains to apply \cite[Thm 3.1]{Calafiore} to express \eqref{eq:dr_chance_constraint}-\eqref{eq:dr_input_constraint} as SOC constraints
\begin{align}
	h_{t,r}^\top \bar{z}_k &\leq 1 - \sqrt{\kappa_\beta} \sqrt{\frac{p_r^x}{1-p_r^x}} \Vert h_{t,r}^\top (\bar{B} \bar{M}_k + \bar{E}) \hat{\Sigma}_N^{1/2} \Vert_2 \label{eq:analytic_state_constraints}\\
	l_{t,s}^\top \bar{v}_k &\leq 1 - \sqrt{\kappa_\beta} \sqrt{\frac{p_s^u}{1-p_s^u}} \Vert l_{t,s}^\top \bar{M}_k  \hat{\Sigma}_N^{1/2} \Vert_2. \label{eq:analytic_input_constraints}
\end{align}
\subsection{Distributionally robust cost function}
\label{sec:cost_fcn}
Similar to the previous section we robustify the cost function \eqref{eq:smpc:cost} against distributional ambiguity. To this end, we reformulate \eqref{eq:smpc:cost} by leveraging the matrices introduced in Section \ref{sec:problem}, \ref{sec:chance_constraints} and define the distributionally robust cost function as
\begin{align}
	&J_k(k) \nonumber \\
	&= \sup_{\mu \in \mathcal{P}}\: \mathbb{E}_{\mu} \left\{ \Vert x_{N|k} \Vert_{P}^2 + \sum_{t=0}^{N-1} \Vert x_{t|k} \Vert_{Q}^2 + \Vert u_{t|k} \Vert_{R}^2  \bigg | x(k)\right\} \nonumber\\
	&= \sup_{\mu \in \mathcal{P}}\: \mathbb{E}_{\mu} \{ d^\top \left( \bar{H}_k^\top \bar{Q} \bar{H}_k + \bar{F}_k^\top \bar{R} \bar{F}_k  \right) d \: | \: x(k)  \} \nonumber \\
	&= \text{tr}\left(\sup_{\mu \in \mathcal{P}}\: \mathbb{E}_{\mu} \{ d d^\top| x(k) \} \big[ \bar{H}_k^\top \bar{Q} \bar{H}_k + \bar{F}_k^\top \bar{R} \bar{F}_k \big]  \right), \label{eq:dr_cost}
\end{align}
where $\bar{H}_k = [\bar{B} \bar{M}_k + \bar{E}, \bar{z}_k]$, $\bar{F}_k = [\bar{M}_k, \bar{v}_k]$, $\bar{Q} = \text{diag}(I_N \otimes Q, P)$, $\bar{R} = I_N \otimes R$. The third equality applies the trace trick for quadratic forms, which can be further simplified with
\begin{align*}
	\hat{\Sigma}^{d}_N \coloneqq \sup_{\mu \in \mathcal{P}}\mathbb{E}_{\mu}\{ dd^\top | x(k) \} =  \scriptsize\begin{bmatrix}
		\kappa_\beta \hat{\Sigma}_N  & 0 \\
		0 & 1
	\end{bmatrix}.
\end{align*}
For the cost $J_k(\cdot)$ we use the convention that the subscript denotes the time on which the expected value is conditioned on, whereas the argument is the closed-loop time instant at which the underlying MPC optimization problem is solved.
\subsection{Terminal constraints}%
We enforce stability of the controller by imposing constraints at the end of the prediction horizon, where we make the following assumption.
\begin{assum}
	\label{assum:terminal_set}
	There exists a terminal controller $\pi_f(z) = K z$ and a terminal set $\mathbb{Z}_f$, such that for all $z \in \mathbb{Z}_f$
	\begin{align*}
		&(A+BK)z \in \mathbb{Z}_f &\\
		& h^\top_r z \leq 1 - \sqrt{\frac{p_r^x}{1-p_r^x}} \left \Vert h_{r}^\top \hat{\Sigma}_\infty^{1/2} \right \Vert_2,\: &r = 1, \ldots, n_r \\
		& l^\top_s K z \leq 1 - \sqrt{\frac{p_r^x}{1-p_r^x}} \left \Vert l_s^\top K \hat{\Sigma}_\infty^{1/2} \right \Vert_2, \: &s = 1, \ldots, n_s 
	\end{align*}
	where $\hat{\Sigma}_\infty = (A+BK) \hat{\Sigma}_\infty (A+BK)^\top + \kappa_\beta E \hat{\Sigma} E^\top$ and $h_r \in \mathbb{R}^{n_x}$, $l_s \in \mathbb{R}^{n_u}$ denote the l.h.s. of the terminal half-space constraints.
\end{assum}
The first condition of Assumption \ref{assum:terminal_set} ensures that the terminal set is invariant for the nominal system under the terminal controller, whereas the second and third conditions enforce the distributionally robust chance constraints for all $z \in \mathbb{Z}_f$ under the worst-case steady-state covariance $\hat{\Sigma}_\infty$.
\begin{remark}
	Assumption \ref{assum:terminal_set} can be ensured with methods proposed \cite[Sec. 2.4.2]{conte2016distributed}, i.e. by using an ellipsoidal terminal set $\mathbb{Z}_f = \{ x \:|\: x^\top P x \leq \alpha \}$ as an $\alpha$-scaled sub level set of the terminal cost function $V_f(x) = \Vert x \Vert_P^2$. All that remains is to choose a scalar $\alpha$ such that the terminal state and input chance constraints (inequality constraints in Assumption \ref{assum:terminal_set}) are satisfied, which can be easily determined as the solution to a linear optimization problem.
\end{remark}

\subsection{Initial condition}
The final and most crucial point in ensuring recursive feasibility is the selection of a suitable initial condition for the MPC optimization problem. Various methods have been proposed in the literature, e.g., Coppens and Patrinos \cite{coppens2021data} require a boundedness assumption on the disturbance to always initialize $x_{0|k} = x(k)$.  
Farina et al. \cite{farina2013probabilistic} remove this boundedness assumption and propose a reset based initialization scheme, that is, whenever possible the feedback strategy (Mode 1) $x_{0|k} = x(k)$ is selected. However, this may lead to infeasibility in \eqref{eq:constraints} due to the unboundedness of $w(k)$. Thus, if Mode 1 is infeasible, the backup strategy $x_{0|k} = z_{1|k-1}$ (Mode 2) is applied. A similar approach was adopted by Hewing et al. \cite{hewing2018stochastic}. 
One downside of the reset-based initialization scheme is the necessity of solving two optimization problems whenever Mode 1 is infeasible. In this work we adopt a recently proposed initialization scheme from \cite{kohler2022recursively}, which interpolates linearly between Mode 1 and Mode 2 resulting in the constraint
\begin{align}
	z_{0|k} &= (1 - \lambda_k) x(k) + \lambda_k z_{1|k-1}, \label{eq:init:mean}
\end{align}
where $\lambda_k \in [0, 1]$. The advantage is that only one optimization problem needs to be solved, where $\lambda_k = 1$ reflects the guaranteed feasible solution (Mode 2) and $\lambda_k = 0$ the feedback strategy (Mode 1).
\subsection{Optimization problem}
At each time instant $k\geq 0$ we solve the following MPC optimization problem
\begin{subequations}
	\label{eq:mpc_optimization}
	\begin{align}
		\!\min_{\bar{v}_k,\bar{M}_k, \lambda_k, z_{0|k}} &  \quad \text{tr}\left(\hat{\Sigma}^{d}_N \big[ \bar{H}_k^\top \bar{Q} \bar{H}_k + \bar{F}_k^\top \bar{R} \bar{F}_k \big]  \right) \label{eq:dr_mpc:cost}  \\
		\text{s.t.} & \quad \bar{z}_k = \bar{A} z_{0|k} + \bar{B} \bar{v}_k  \\
		& \quad \eqref{eq:analytic_state_constraints}, \eqref{eq:analytic_input_constraints}, \eqref{eq:init:mean}, \lambda_k \in [0, 1]\\
		& \quad z_{N|k} \in \mathbb{Z}_f. \label{eq:dr_mpc:terminal_constraints}
	\end{align}
\end{subequations}
The solution to problem \eqref{eq:mpc_optimization} is the optimal SADF pair $(\bar{v}_k^*, \bar{M}_k^*)$ and the nominal states $\bar{z}_k^*$. To obtain the control input at time $k$, we recall \cite[Thm. 1]{zhang2020stochastic}, which establishes an equivalence between the SADF parameterization \eqref{eq:disturbance_feedback} and the state feedback parameterization. By linear superposition, we can thus establish also an equivalence to the error-feedback (EF) parameterization $u_{t|k} = g_{t|k} + \sum_{i=0}^{t} K_{t-i|k} (x_{i|k} - z_{i|k})$. In other words, the state and input trajectories $(\bar{x}_k, \bar{u}_k)$ resulting from SADF with $(\bar{v}_k^*, \bar{M}_k^*)$ are equivalent to the ones obtained from EF with $(\bar{g}_k^*, \bar{K}_k^*)$, where  %
\begin{align*}
	{\tiny
		\bar{K}_k \coloneqq \begin{bmatrix}
			K_{0|k} & 0 & \ldots & 0 & 0 \\
			K_{1|k} & K_{0|k} & \ldots & 0 & 0 \\
			\vdots & \ddots & \ddots & \vdots & 0 \\
			K_{N-1|k} & \ldots & K_{1|k} & K_{0|k} & 0 
		\end{bmatrix}, \bar{g}_k \coloneqq \begin{bmatrix}
			g_{0|k} \\
			g_{1|k} \\
			\vdots \\
			g_{N-1|k}
		\end{bmatrix}.
	}
\end{align*}
Similar to \cite{zhang2020stochastic}, the optimal pair $(\bar{g}_k^*, \bar{K}_k^*)$ is obtained by
\begin{subequations}
	\begin{align}
		\bar{K}_k^* &= (I + \bar{M}_k^* \bar{E}^\dagger \bar{B})^{-1} \bar{M}^*_k \bar{E}^\dagger \label{eq:K_trafo}\\
		\bar{g}_k^* &= (I + \bar{M}_k^* \bar{E}^\dagger \bar{B})^{-1} ( \bar{v}^*_k - \bar{M}_k^* \bar{E}^\dagger A z^*_{0|k}), \label{eq:g_trafo}
	\end{align}
\end{subequations}
while the input to system \eqref{eq:dynamics} is defined with the EF parameterization at time $t=0$, resulting in 
\begin{align}
	u(k) = u_{0|k} = g^*_{0|k} + K_{0|k}^* (x(k) - z_{0|k}^*). \label{eq:closed_loop:input}
\end{align}
\begin{remark}
	\label{rem:chance_constraint}
	Note that the chance constraints \eqref{eq:dr_chance_constraint} - \eqref{eq:dr_input_constraint} depend on the information available at time $k$. In view of the initial condition \eqref{eq:init:mean}, this implies that whenever the MPC problem \eqref{eq:mpc_optimization} is feasible with $\lambda_{k} = 0$, the probability operator in \eqref{eq:dr_chance_constraint} - \eqref{eq:dr_input_constraint} is conditioned on time $k$, resulting in closed-loop constraint satisfaction, while for $\lambda_{k} \in (0, 1]$ the constraints are verified in prediction, i.e. conditioned on the last time instant $k-\tau$ when problem \eqref{eq:mpc_optimization} was feasible with $\lambda_{k-\tau} = 0$. 
	Note that the same conditioning appears in the expectation operator of the cost function  \eqref{eq:dr_cost}.
	
	By leaving $\lambda_{k}$ un-penalized in the objective function \eqref{eq:dr_mpc:cost}, we mimic a so-called hybrid scheme \cite{farina2016stochastic} with the intention to minimize the open-loop cost despite feasibility of $x(k)$. 
	This can lead to an increase in constraint violations in presence of unmodeled disturbances, as we will demonstrate in Section \ref{sec:numerical_example:disturbances}. However, adding an additional penalty term $c \lambda_k^2$ with $c>0$ to the objective function \eqref{eq:dr_mpc:cost} causes the MPC controller to favor feedback initialization $z_{0|k} = x(k)$ with the intention of introducing as much feedback as possible into the constraints, i.e. conditioning the probability operator in \eqref{eq:dr_chance_constraint} - \eqref{eq:dr_input_constraint} on time $k$ as often as possible. A possible drawback is the degradation of transient closed-loop performance, since the initial state cannot be freely chosen and the optimization problem therefore has fewer degrees of freedom, see Section \ref{sec:numerical_example:disturbances} for a numerical comparison.
\end{remark}
\section{Theoretical properties}
\label{sec:theoretical}
\subsection{Recursive feasibility}
\begin{proposition}
	\label{prop:recursive}
	Let Assumption \ref{assum:terminal_set} hold. If at time $k=0$ the MPC optimization problem \eqref{eq:mpc_optimization} admits a feasible solution with $\lambda_{0} = 0$, then it is recursively feasible for all $k \geq 0$.
\end{proposition}
\begin{proof}
	Suppose that at time $k$ problem \eqref{eq:mpc_optimization} is feasible with $\lambda^*_k = 0$, $(\bar{v}_k^*, \bar{M}_k^*)$, $\bar{z}_k^*$ and equivalently with the EF parameterized input
	$\bar{u}_k = \bar{g}^*_k + \bar{K}^*_k (\bar{x}_k - \bar{z}_k^*)$ with $(\bar{g}_k^*, \bar{K}_k^*)$ due to \cite[Thm. 1]{zhang2020stochastic}. Now we construct the usual shifted candidate sequence $\tilde{u}_{t|k+1} = u_{t+1|k}$ for $t = 0, \ldots, N-2$ and append the terminal controller $\tilde{u}_{N-1|k+1} = K z^*_{N|k}$. The shifted mean states and controller gains satisfy {\small$(\tilde{z}_{t|k+1}, \tilde{K}_{t|k+1}) = (z^*_{t+1|k}, K^*_{t+1|k})$ for $t = 0, \ldots, N-1$} appended with {\small$(\tilde{z}_{N|k}, \tilde{K}_{N|k}) = ((A+BK) z^*_{N|k}, K)$}. Recursive feasibility is then a consequence of Assumption \ref{assum:terminal_set}. By stacking the shifted candidate sequences into the corresponding matrix and vector form, we obtain the triplet $(\tilde{\bar{g}}_{k+1}, \tilde{\bar{K}}_{k+1}, \tilde{\bar{z}}_{k+1})$. A feasible input pair $(\bar{v}_{k+1}, \bar{M}_{k+1})$ for problem \eqref{eq:mpc_optimization} is then simply found by \cite[eq. (24)]{zhang2020stochastic}, i.e.
	{\small
		\begin{align*}
			\bar{M}_{k+1} &= \tilde{\bar{K}}_{k+1} (I - \bar{B} \tilde{\bar{K}}_{k+1})^{-1} \bar{E} \\
			\bar{v}_{k+1} &= \tilde{\bar{K}}_{k+1} (I - \bar{B} \tilde{\bar{K}}_{k+1})^{-1} (\bar{A} \tilde{z}^*_{0|k} + \bar{B} \tilde{\bar{g}}_{k+1} ) + \tilde{\bar{g}}_{k+1}
		\end{align*}
	} \normalsize
	with $\lambda_{k+1} = 1$. This concludes the proof.
\end{proof}
\subsection{Convergence}
The following theorem establishes a quadratic stability result of the closed-loop system \eqref{eq:dynamics} under control law \eqref{eq:closed_loop:input}. By adding an additional penalty term $c \lambda_k^2$ to the cost function \eqref{eq:dr_mpc:cost} (cf. Remark \ref{rem:chance_constraint}), the performance bound \eqref{eq:average_bound} changes to $\kappa_\beta \text{tr}( P E\hat{\Sigma}E^\top) + c$. 
\begin{theorem}
	\label{lem:cost_decrease}
	Let Assumption 1 hold and choose $\beta \in (0,1)$, $\epsilon \in (0, 0.5)$ and $N_s$, such that \eqref{eq:thm:condition} holds true and let $\hat{\Sigma}$ be the corresponding empirical covariance matrix. Suppose that at time $k=0$ there exists a feasible solution to problem \eqref{eq:mpc_optimization}. Then, for all $k \geq 0$ the optimal cost $J^*_{k}(k+1)$ satisfies
	\begin{multline*}
		J^*_{k}(k+1) - J^*_k(k) \\ 
		\leq - \mathbb{E}\{\Vert x(k) \Vert_Q^2 + \Vert u(k) \Vert_R^2 | x(k)\} + \kappa_\beta \text{tr}( PE \hat{\Sigma}E^\top).
	\end{multline*}
	Furthermore, with a probability of at least $1-\beta$ the closed-loop system achieves the following asymptotic average bound
	{\small
		\begin{multline}
			\label{eq:average_bound}
			\mathbb{P} \bigg( \lim_{T \rightarrow\infty} \frac{1}{T} \sum_{k = 0}^{T-1} \mathbb{E}_{\mu^*} \{ \Vert x(k) \Vert_Q^2 + \Vert u(k) \Vert_R^2  | x(0) \} \\
			\leq \kappa_\beta \text{tr}( P E\hat{\Sigma}E^\top) \bigg) \geq 1-\beta.
	\end{multline} }
\end{theorem}
\begin{proof}
	We establish an expected cost decrease condition in case of $\lambda_{k+1} = 1$. First, consider the cost function \eqref{eq:dr_cost}, which, due its quadratic form and in view of Proposition \ref{prop:recursive} can equivalently be written as $J_k(k) = J^m(\bar{z}_{k}, \bar{g}_{k}) + J^v(\kappa_\beta \hat{\Sigma}, \bar{K}_{k})$, where the mean and variance part satisfy
	{\small
		\begin{align*}
			J^m(\bar{z}_{k}, \bar{g}_{k}) &= \Vert z_{N|k} \Vert_{P}^2 + \sum_{t=0}^{N-1} \Vert z_{t|k} \Vert_{Q}^2 + \Vert g_{t|k} \Vert_{R}^2 \\
			J^v(\kappa_\beta \hat{\Sigma}, \bar{K}_{k}) &= \text{tr}(P \hat{\Sigma}^x_{N|k}) + \sum_{t=0}^{N-1} \text{tr}(Q \hat{\Sigma}^x_{t|k} + R \hat{\Sigma}^u_{t|k})
		\end{align*}
	}
	with $\hat{\Sigma}^x_{t+1|k} = (A+BK_{t|k}) \hat{\Sigma}^x_{t|k} (A+BK_{t|k})^\top + \kappa_\beta E \hat{\Sigma} E^\top$ and $\hat{\Sigma}^u_{t|k} = \sum_{i = 0}^{t} K_{t-i|k} \hat{\Sigma}^x_{i|k} K^\top_{t-i|k}$. Using the feasible solution from Proposition \ref{prop:recursive} we can argue by optimality that
	{\small
		\begin{align*} 
			&J^*_k(k+1) \overset{\lambda_{k+1} = 1}{\leq} J^m(\bar{z}_{k+1}, \bar{g}_{k+1}) + J^v(\kappa_\beta \hat{\Sigma}, \bar{K}_{k+1}) \\
			&= J^m(\bar{z}_{k}, \bar{g}_{k}) - \Vert z_{0|k} \Vert_Q^2 - \Vert g_{0|k} \Vert_R^2 \\
			&\quad + \Vert z_{N|k} \Vert_Q^2 + \Vert K z_{N|k} \Vert_R^2 - \Vert z_{N|k} \Vert_P^2 + \Vert A_K z_{N|k} \Vert_P^2 \\
			&+ J^v(\kappa_\beta \hat{\Sigma}, \bar{K}_{k}) - \text{tr}( [Q+K_{0|k}^\top R K_{0|k}] \hat{\Sigma}^x_{0|k} ) \\
			&\quad + \text{tr}( [Q+K^\top R K] \hat{\Sigma}^x_{N|k} + P A_K \hat{\Sigma}^x_{N|k} A_K^\top \\
			&\quad + P E \kappa_\beta \hat{\Sigma} E^\top - P \hat{\Sigma}^x_{N|k}) \\
			&\overset{\eqref{eq:lyapunov}}{\leq} J^m(\bar{z}_{k}, \bar{g}_{k}) + J^v(\kappa_\beta \hat{\Sigma}, \bar{K}_{k})  - \Vert z_{0|k} \Vert_Q^2 - \Vert g_{0|k} \Vert_R^2 \\
			& - \text{tr}( [Q+K_{0|k}^\top R K_{0|k}] \hat{\Sigma}^x_{0|k}) + \kappa_\beta \text{tr}(P E \hat{\Sigma} E^\top)  \\
			&= J_k^*(k) - \mathbb{E}\{ \Vert x(k) \Vert_Q^2 + \Vert u(k) \Vert_R^2 | x(k) \} + \kappa_\beta\text{tr}(PE\hat{\Sigma} E^\top),
		\end{align*}
	}
	where $A_K = A + BK$. To achieve the asymptotic cost bound we follow standard arguments in stochastic MPC \cite{hewing2018stochastic}, i.e.
	{\small 
		\begin{align*}
			0 &\leq \lim_{T\rightarrow\infty} \frac{1}{T} \big(J^*_{0}(k) - J^*_0(0) \big) \\
			& \leq  \lim_{T\rightarrow\infty} \frac{1}{T} \sum_{k = 0}^{T-1} -\mathbb{E}_{\mu^*} \{ \Vert x(k) \Vert_Q^2 + \Vert u(k) \Vert_R^2 | x(0)\}\\
			&\quad +  \kappa_\beta \text{tr}( P E \hat{\Sigma} E^\top),
		\end{align*}
	}\normalsize whereas the probability bound \eqref{eq:average_bound} follows by definition of the ambiguity set \eqref{eq:ambiguity_set}, i.e. $\mathbb{P}( \Sigma \leq \kappa_\beta \hat{\Sigma}) \geq 1-\beta$.
\end{proof}

\section{Numerical example}
\label{sec:example}
In this section, we carry out a numerical example. We consider the following system
\begin{align*}
	x(k+1) =\begin{bmatrix}
		1 & 1 \\ 0 & 1
	\end{bmatrix} x(k) + \begin{bmatrix}
		1 \\0.5
	\end{bmatrix} u + \begin{bmatrix}
		1 & 0 \\ 0 & 1
	\end{bmatrix} w(k),
\end{align*}
where $w(k) \sim \mathcal{N}(0, \Sigma)$ with $\Sigma = 0.01^2 I$. For the ambiguity radius we select $\beta = 0.05$, $\epsilon = 0.0428$. Since $w(k)$ is a zero-mean Gaussian it follows that $\xi = \Sigma^{1/2} w$ is sub-Gaussian with variance $\sigma^2 = 1$. According to Proposition \ref{prop:covariance} we require $N_s \geq 516$ samples to give the guarantee that $\mathbb{P}(\mu^* \in \mathcal{P}) \geq 1-\beta$. 
For the MPC cost function we select the weighting matrices $\tiny Q = \begin{bmatrix}
	10 & 0 \\ 
	0 & 10
\end{bmatrix}$, $\tiny R = 1$ and $\tiny P = \begin{bmatrix}
	20.5988  &  5.9161 \\
	5.9161 &  14.2284
\end{bmatrix}$. We impose a single chance constraint $\mathbb{P}( x_2(k) \leq 1) \geq p_x$ and use an ellipsoidal terminal set $\mathbb{Z}_f = \{z \:|\: z^\top P z \leq \alpha \}$, where $\alpha = 0.5293$ is obtained from $N_s = 517$ samples. We keep $\alpha$ constant for each experiment and select a prediction horizon of $N = 10$. Note that the choice of $\alpha$ is quite conservative, i.e. for $N_s = 10^3$ the resulting terminal set is already $20.4$ times larger, whereas under exact moment information we could enlarge the terminal set about $21.9$ times.
\subsubsection{Performance and constraint satisfaction}
Starting at an initial condition $x(0) = [6, 0]^\top$ we performed $10^3$ Monte-Carlo runs for different sample sizes $550 \leq N_s \leq 10^6$ of the closed-loop system. As we can see in Figure \ref{fig:cost_decrease} (left) the expected cost converges asymptotically to the optimal cost derived with exact moment information as the sample size $N_s$ increases. As for the chance constraints, it can be seen in Figure \ref{fig:cost_decrease} (right) that as the number of samples increases, the controller becomes more confident to operate closer to the constraint. In Table \ref{table} we compare for different prescribed probability levels $p_x$ and sample sizes $N_s$ the achieved empirical constraint satisfaction rate averaged over $10^4$ Monte-Carlo runs. The discrepancy between the prescribed and empirical satisfaction rate can be reduced by introducing more knowledge of the underlying probability distribution, e.g. as in stochastic MPC. Finally, the MPC optimization problem \eqref{eq:mpc_optimization} is reliably solved in an average of $6$ milliseconds on a desktop PC with an Intel Core i7-9700k processor, Yalmip \cite{lofberg2004yalmip} and MOSEK \cite{aps2019mosek}.
\begin{center}
	\captionof{table}{Effect of sample size $N_s$ on the worst-case empirical probability of satisfying the constraint $\mathbb{P}(x_2 \leq 1) \geq p_x$.}
	\begin{tabular}{l|llll}
		$p_x$  & $N_s = 520$ & $N_s = 800 $ & $N_s = 10^5$  & $N_s = 10^6$\\ \hline
		$0.7$  & $100 \% $   & $99.25 \%$   & $86.95 \%$    & $ 85.99 \%$ \\
		$0.8$  & $100 \% $   & $99.95 \%$    & $93.81 \%$   & $ 93.29 \%$ \\
		$0.9$  & $100 \% $   & $100 \%$     & $99.17 \%$    & $ 98.83 \%$ \\
	\end{tabular} 
	\label{table}
\end{center}
\begin{figure}[t]
	\centering
	\includegraphics[width=1\linewidth]{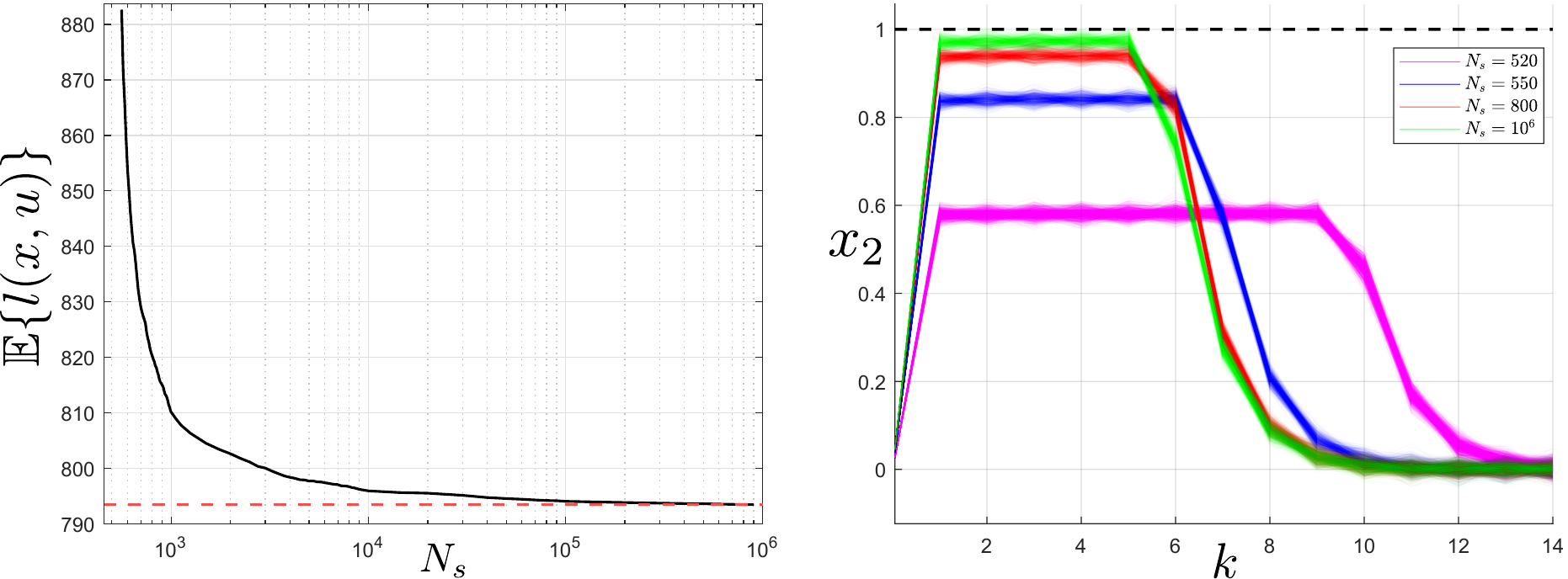}
	\caption{(\textbf{Left}) Expected closed-loop cost $l(x,u) = \sum_{k=1}^{15} \Vert x(k) \Vert_Q^2 + \Vert u(k) \Vert_R^2$ for different sample sizes $N_s$ computed over $10^3$ Monte-Carlo simulations (black) and optimal cost under exact moment information (red). (\textbf{Right}) Closed-loop trajectories for different $N_s$ with $p_x = 0.9$. The black dotted line denotes the constraint $x_2 \leq 1$.}
	 
	\label{fig:cost_decrease}%
\end{figure}
\subsubsection{Unmodeled disturbances}
\label{sec:numerical_example:disturbances}
In the following we investigate the benefits of adding a penalty term for $\lambda_k$ to the objective function \eqref{eq:dr_mpc:cost}. To this end we keep the same simulation setup as before and introduce an unmodeled larger disturbance at time step $k = 5$ with $w(5) \sim \mathcal{N}(0,6 \Sigma)$. We add $c \lambda_k^2$ to the objective function \eqref{eq:dr_mpc:cost} to force the MPC optimization problem to prefer the feedback initialization over open-loop cost reduction (Remark \ref{rem:chance_constraint}) and opt to satisfy the chance constraint with $70 \%$ probability.
\begin{center}
	\captionof{table}{Comparison of different controller configurations.}
	\begin{tabular}{l|llll}
		Method  				    & $c = 0$ 		& $c = 10$ 	   & $c = 10^3$ & $c = 10^6$  \\ \hline
		$\mathbb{E}\{l(x,u)\}$  	& $783.20 $     & $784.47$     & $784.50$ 	& $784.74$    \\
		$\mathbb{P}(x_2(5) \leq 1)$ & $68.76 \% $   & $73.37 \%$   & $73.91 \%$ & $ 75.08\%$  \\
	\end{tabular} 
	\label{table2}
\end{center}
Table \ref{table2} reveals that penalization of $\lambda_k$ increases the constraint satisfaction rate by sacrificing transient closed-loop performance compared to the unpenalized case $c = 0$. Additionally, for $c > 0$ the chance constraint is empirically verified, whereas $c = 0$ violates the prescribed level of $70 \%$.
\section{Conclusion}
We have presented a DR-MPC framework for linear systems with additive disturbances under moment-based ambiguity sets, providing guarantees on closed-loop performance and recursive feasibility. The chance constraints are replaced with distributionally robust chance constraints in form of SOC constraints, whereas the cost function is minimized subject to the worst-case distribution across the ambiguity set. We used a simple numerical example to highlight the properties of the resulting controller.

\end{document}